\documentclass[10pt,a4paper]{article}
\usepackage[latin1]{inputenc}
\usepackage{amsmath}
\usepackage{amsfonts}
\usepackage{amssymb}
\usepackage{graphicx}
\usepackage{amsthm}
\usepackage{bbm}
\usepackage{nicefrac}
\usepackage{theoremref} 
\usepackage[title]{appendix}
\numberwithin{equation}{section}
\newtheorem{Thm}{Theorem}
\newtheorem{Lemma}[Thm]{Lemma}
\newtheorem{cor}[Thm]{Corollary}

\newtheorem{definition}[Thm]{Definition}
\newtheorem*{rmk}{Remark}
\newtheorem*{reminder}{Reminder}
\newcommand{\norm}[1]{\left\| #1 \right\|}
\newcommand{\abs}[1]{\left|#1\right|}
\newcommand{\myLtwo}{\mathnormal{L}^2 \left( \mathbbm{1}_{\left[ -2,2 \right]}(x) \left( 4-x^2 \right)^\frac{3}{2}  \mathrm{d}x \right)}

\begin{document}
	\title{A convergent $\frac{1}{N}$ expansion for GUE }
	\author{Offer Kopelevitch\textsuperscript{1}}
	\footnotetext[1]{School of Mathematical
Sciences, Tel Aviv University, Tel Aviv, 69978, Israel. E-mail: offer358@walla.com. Supported in part by the European Research Council start-up grant 639305 (SPECTRUM)}
	\maketitle
	\begin{abstract}
We show that the  asymptotic $1/N$ expansion for the averages of linear statistics of the GUE  is convergent when the test function is an entire function of order two and finite type. This allows to fully recover the mean eigenvalue density function for finite $N$ from the coefficients of the expansion thus providing a resummation procedure. As an intermediate result we compute the bilateral Laplace transform of the GUE reproducing kernel in the half-sum variable, generalizing a formula of Haagerup and Thorbj\o rnsen.
\end{abstract}
	\section{Introduction}
	
	We study the spectral asymptotics of the random matrices from the Gaussian Unitary Ensemble (GUE). By a GUE($N$) matrix we shall mean a random hermitian matrix $X_N $ whose  entries are jointly independent random variables up to the condition $ X_{jk} = \bar{X}_{kj} $ , the diagonal entries are distributed as $ N(0,\frac{1}{N}) $ and for $j<k $ ,  $ \Re(X_{jk}) $ and $ \Im(X_{jk}) $ are independent with distribution $ N(0, \frac{1}{2N}) $. The eigenvalues of this matrix,  $ \lambda_1 \leq \cdots \leq \lambda_N $, form a determinantal point process with kernel \[ K_N (x,y) = e^{ -\frac{N}{4}(x^2 + y^2)}  \frac{\tilde{h}_{N-1}(x) \tilde{h}_N(y) - \tilde{h}_N(x) \tilde{h}_{N-1}(y)}{x-y} ~, \] where $\tilde{h}_k $ are the normalized Hermite polynomials (see Appendix \ref{HerProp}). This means that the correlation functions of the process are expressed as determinants constructed from $K_N$.  We only state the  formula which is important in the sequel: the mean eigenvalue density $p_N(x)$ function of the GUE, defined by
	 \[ \frac{1}{N} \mathbb{E} \sum_{j=1}^{N} \mathbbm{1}_{a<\lambda_j<b} = \int_{a}^{b} p_N(x) \mathrm{d}x~, \]
is given by  the $1 \times 1$ determinant $ p_N(x) =\frac{1}{N}K_N(x,x) $. These and further properties of the GUE are discussed, for example,  in \cite{Book-Intro, Intro}.

	Wigner \cite{WigSemiCircle} proved that for a large family of random matrix ensembles with independent entries (including the GUE)  the eigenvalue distribution tends to the semi-circle law as the size of the matrix tends to infinity, i.e., 
	\[ \frac{1}{N}\sum_{j=1}^{N} \mathbbm{1}_{a<\lambda_j<b} \underset{N\to \infty}{\longrightarrow } \frac{1}{2\pi} \int_{a}^{b} \sqrt{4-x^2}\mathbbm{1}_{(-2,2)}(x) \mathrm{d}x ~, \] 
and the same holds with the expectation on the left-hand side. We are interested in the corrections to the semi-circle for a finite $N$. 

 Ercolani and McLaughlin \cite{EM} proved that for any function $f$ from the space of $C^\infty $ functions with at most polynomial growth (denote this space by $C^\infty_\mathrm{p}$) and for $X_N$ from some class of random matrices with unitary-invariant distribution (including the GUE) there exists a sequence $(\alpha_j(f))_{j\in\mathbb{N}_0}$ of complex numbers such that for any $k\in \mathbb{N}_0$,
	
	\begin{equation} \label{AsymptoticExp}
	 \frac{1}{N}\mathbb{E} \left\{ \mathrm{Tr} \left(f\left(X_N\right)\right) \right\} = \int f(t) p_N(t) \mathrm{d}t = \sum_{j=0}^{k} \frac{\alpha_j(f)}{N^{2j}} + O\left(N^{-2k-2}\right) ~.
	\end{equation}
	
	Haagerup and Thorbj\o rnsen \cite{HT} provided an alternative proof of \eqref{AsymptoticExp} in the case of the GUE using a differential equation for the mean density function of the GUE, $p_N$. (The differential equation, see \eqref{eq:diffeq} below, was proved by G\"otze and Tikhomirov \cite{DiffEq} building on earlier work of Haagerup and Thorbj\o rnsen \cite{large eig}.) Furthermore, they showed that 
	\begin{equation} \label{DistExp}
	\alpha_j(f) = \frac{1}{2\pi} \int_{-2}^{2} \left[T^jf\right](x) \sqrt{4-x^2} \, \mathrm{d}x 
	\end{equation}
	for a certain linear operator $T\colon C^\infty_p \to C^\infty_p  $ (see Section \ref{Operator}). They also found an explicit expression for the error term: 
	\begin{equation}
	\int f(t) p_N(t) \mathrm{d}t = \sum_{j=0}^{k} \frac{\alpha_j(f)}{N^{2j}} + \frac{1}{N^{2k+2}} \int_{-\infty}^{\infty} \left[T^{k+1}f\right](x) p_N(x)\mathrm{d}x ~.
	\end{equation}
	
	\medskip\noindent
	It might seem that the series 
	\begin{equation}\label{eq:series} p_N \sim \sum_{j=0}^\infty \frac{\alpha_j}{N^{2j}} \end{equation}
	cannot be convergent since the support of $ \alpha_j $ (in the sense of distributions) is the interval $ [-2,2] $, whereas the support of $p_N$ is the entire line. \\Another reason to believe that such series cannot be convergent comes from the integral representations for $p_N$. We show one derived from equations (3.28) and (1.3) of \cite{Dis} (see \cite{IN,Sh} for alternative ones):
	\[ p_N(t) = - \frac{N}{2\pi^2} \Im \int_{(\mathbb R + i) \times \mathbb{R}} \mathrm da \, \mathrm  db \, e^{-\frac{N}{2} (a^2 + b^2)} 
	\frac{(t-ib)^N}{(t-a)^N}  \, a \left[ 1 - \frac{1}{(t-a)(t-ib)} \right]~. \]
Computing the large $N$ asymptotics via the saddle-point approximation, one finds two relevant saddle points, one of which
yields the terms corresponding to $\alpha_j$, while the other one yields non-vanishing oscillating corrections, the first of which is of order $\frac{1}{N}$ (see \cite{Dis}). Both arguments seem to indicate that there are non-perturbative corrections to (\ref{eq:series}). 

\medskip
 However, Harer and Zagier \cite{HZ} showed that the moments of the GUE can be written as a finite power series in $\frac{1}{N^2}$, i.e, when $X_N $ is a $ \mathrm{GUE}\left(N\right) $ matrix,
	\begin{equation}\label{eq:HZ}
	\int_{-\infty}^\infty t^n p_N(t) \mathrm{d}t = \sum_{0\leq k \leq \frac{n}{2}} \frac{\epsilon_k(n)}{N^{2k}} ~,
	\end{equation}
	where $\epsilon_0(n) $ are the moments of the semicircle. The moment problem for $p_N(t)$ is determinate,
	therefore (\ref{eq:HZ}) provides a resummation procedure for (\ref{eq:series}); however, this procedure is 
	not very explicit and does not seem suitable for taking scaling limits. 
	
	This raises the question for which functions $f$ the integral  
	$\int f(t) p_N(t) \mathrm{d}t$ can be written as a convergent power series in $\frac{1}{N^2}$. Our main goal is to answer this question, see \thref{mainThm} below.
	
	\medskip
	The paper is organized as follows:  In \thref{FourierTrans} of Section~\ref{Laplace} we generalize a result of Haagerup and Thorbj\o rnsen \cite{large eig}, who computed the bilateral Laplace transform of the density function (i.e., $ \frac{1}{N} \mathbb{E} \left\{ \mathrm{Tr} \left[e^{sX_N}\right] \right\} $). We compute the bilateral Laplace transform of $ K_N(\lambda_+ +\lambda_-, \lambda_+ - \lambda_-) $ with respect to $ \lambda_+ $ (setting $\lambda_-=0$ recovers the result of \cite{large eig}).  The expression has a peculiar hyperbolic symmetry between $\lambda_-$ and the parameter of the Laplace transform, which we believe to be interesting, however in the sequel we only use the original result of \cite{large eig}. 
	
	In Section \ref{Laplace Expansion} we show by an explicit computation that the bilateral Laplace transform of the density function can be expressed as a convergent power series in $\frac{1}{N} $ , i.e.,  
	\begin{equation} \label{LapExpStatement} 
	\int_{-\infty}^\infty e^{st} p_N(t)\,\mathrm{d}t=\sum_{k=0}^{\infty} \frac{c_k(s)}{N^{2k}} ~,
	\end{equation} where $c_0(s) $ is the bilateral Laplace transform of the semi-circle law, and more generally $c_k(s)$ is the bilateral Laplace transform of the distribution $ \alpha_k $. The formula \eqref{LapExpStatement} is valid for all $s\in \mathbb C $ and provides a resummation procedure for \eqref{eq:series}:
	\begin{equation} \label{Fourier resummation}
	 p_N = \left(\sum_{j=0}^\infty \frac{\widehat{\alpha}_j}{N^{2j}}\right)^\vee
	\end{equation}
	(where $\bullet^\wedge$ and $\bullet^\vee$ denote the direct and inverse Fourier transform, which are well defined since $ p_N(t) $ decays as a polynomial times a gaussian as $t \to \infty $ ).
	
	In Section \ref{Operator} we define the operator $T$ from \eqref{DistExp} on the space of entire functions of order two and finite type: \[ \left\{ f \, \text{entire} \,\big| \, \sigma_f = \limsup_{\abs{z} \to \infty } \frac{ \log \abs{f(z)} }{\abs{z}^2} < \infty \right\} ~,  \]  and show that the operator is bounded with respect to appropriate norms. In Section \ref{proof} we use that to show that the asymptotic series (\ref{eq:series}) is a convergent series of functionals on this space. More precisely, we prove
	
	\begin{Thm} \thlabel{mainThm}
		If $f$ is an entire function of order two and finite type $\sigma_f$, then for a $ \mathrm{GUE}(N) $ random matrix $X_N $ with $N > N_0(\sigma_f)$  one has a convergent expansion$\colon$  
		\[ \frac{1}{N} \mathbb{E} \{ \mathrm{Tr}\left( f(X_N)\right) \} = \int f(t) p_N(t) \, \mathrm{d}t =\frac{1}{2\pi}  \sum\limits_{k=0}^{\infty} \frac{1}{N^{2k}} \int_{-2}^{2} \left[T^kf\right] (t) \sqrt{4-t^2}\, \mathrm{d}t~. \]
	\end{Thm}

\begin{rmk} The condition in \thref{mainThm} is sharp in the following sence. For entire functions of order two and maximal type the expansion might diverge, and in fact $\frac{1}{N} \mathbb{E} \{ \mathrm{Tr}\left( f(X_N)\right) \}$ may be infinite for all $N$. This follows from the estimate
\begin{equation*} \mathbb{P} \left( \lambda_N \geq R \right)  \geq \mathbb{P} \left( X_{11} \geq R \right)  = \int_R^\infty e^{-\frac{Nx^2}{2}} \frac{\sqrt{N}}{\sqrt{2\pi}} \, \mathrm{d}x  \geq \frac{\sqrt{N}}{\sqrt{2\pi}} e^{-\frac{N(R+1)^2}{2}} ~. 
\end{equation*}
 \end{rmk}
 
 Equation \eqref{Fourier resummation} and more generally \thref{mainThm} show that divergent $\frac{1}{N}$ expansions can be sometimes resummed, at least, in the case of the mean eigenvalue density of the GUE. It would be
 interesting to extend this to other spectral characteristics and to other random matrix ensembles. The question could be asked for the unitary invariant matrix 
 ensembles, see \cite{APS} and references therein, and for $\beta$-ensembles, see \cite{BG} and references therein, as well
 as for many other ensembles that admit formal expansions.
 We also hope that convergent $\frac{1}{N}$ expansions could be a useful tool for studying various $N \to \infty$ limits, in particular, local eigenvalue statistics.

\paragraph{Acknowledgement} I am grateful to my supervisor, Sasha Sodin, for his continuous help and advice. I am also would like to thank Yan Fyodorov for helpful discussions and for a suggestion leading to the current version of \thref{FourierTrans}. I would like to thank Andrei Iacob for copy editing the paper.
	\section{Bilateral Laplace transform of the GUE} 
	
	The aim of this section is to calculate the bilateral Laplace transform of the density function of the GUE, $ p_N(\lambda) = \frac{1}{N} K_N (\lambda,\lambda) $, and to deduce from it a differential equation for $p_N$. To do so, we will calculate the bilateral Laplace transform of 
	\[  K_N(\lambda_+ + \lambda_- , \lambda_+ -\lambda_-) \] with respect to the variable $\lambda_+ $ . This will give us beside the differential equation for $p_N(\lambda) $, a symmetry between the parameter of the Laplace transform $s$ and $N\lambda_- $, which can be thought as the distance between eigenvalues normalized such that the mean distance between consecutive eigenvalues is of order one. 
	
	\subsection{Bilateral Laplace transform of the off-diagonal kernel} \label{Laplace}
	
	\begin{Thm} \thlabel{FourierTrans}
		The bilateral Laplace transform of the $\mathrm{GUE}$ kernel $ K_N\left( \lambda_+ + \lambda_- , \lambda_+ - \lambda_- \right) $ with respect to the variable $ \lambda_+ $ is given by 
		\[
		\int_{-\infty}^{\infty} e^{s\lambda_+} K_N(\lambda_+ +\lambda_-,\lambda_+ - \lambda_-) \mathrm{d}\lambda_+ =  Ne^{-\frac{N}{2}\lambda_-^2 +\frac{s^2}{2N} } {}_1F_1\left(1-N ; 2 | N\lambda_-^2 - \frac{s^2}{N} \right) 
		\]
	for any $s \in \mathbb C$. \end{Thm}
	
	\begin{reminder}
		The confluent hypergeometric function is defined by 
		\begin{equation}\label{eq:confl}
		{}_1F_1(a ; c | x) = \sum_{k=0}^{\infty} \frac{a (a+1) \cdots (a+k-1)}{ c (c+1) \cdots (c+k-1) \cdot k!} x^k ~,
		\end{equation}
		is an entire function of $x$,  and satisfies the differential equation 
		\begin{equation}
		xy''+(c-x)y' -ay =0 ~.
		\end{equation}
	\end{reminder}
	
	\begin{rmk}
		\thref{FourierTrans} is a generalization of a result by Haagerup and Thorbj\o rnsen \textup{\cite{large eig}}, who showed that for $\lambda_-=0$ 
	
		\begin{equation} \label{LaplaceExpression}
		\int_{-\infty}^{\infty} e^{s\lambda} K_N(\lambda,\lambda) \mathrm{d}\lambda = N e^{ \frac{s^2}{2N} } {}_1F_1\left(1-N ; 2\, | \, - \frac{s^2}{N} \right)
		\end{equation}
		
	\end{rmk}
	\begin{proof}
		
		For simplicity of computation we will calculate the bilateral Laplace transform of the kernel for real $s$ and infer the theorem by analytic continuation. \\
		Recall that \[ K_N(\lambda_1,\lambda_2) = e^{ -\frac{N}{4} \left(\lambda_1^2 + \lambda_2^2 \right)} \frac {\tilde{h}_{N-1}(\lambda_1)\tilde {h}_{N}(\lambda_2) - \tilde{h}_{N}(\lambda_1)\tilde{h}_{N-1}(\lambda_2) }{\lambda_1 - \lambda_2} ~, \]  where $ \tilde{h}_k (x) = \left[k! N^k \sqrt{\frac{2 \pi}{N}}\right]^{-\nicefrac{1}{2} } h_k (x) $ are the normalized Hermite polynomials and $ h_k (x) = (-1)^k e^{N\frac{x^2}{2}} \frac{\mathrm{d}^k}{\mathrm{d}x^k} \left( e^{-N\frac{x^2}{2}} \right) $ are the Hermite polynomials (we review their properties in Appendix \ref{HerProp}). Then
		\[\begin{split}
		 &e^{s\lambda_+}K_N(\lambda_+ +\lambda_-,\lambda_+ -\lambda_-) \\ &=   e^{s\lambda_+ -\frac{N}{2}(\lambda_+^2+\lambda_-^2) } \,\, \frac {\tilde{h}_{N-1}(\lambda_+ +\lambda_-)\tilde {h}_{N}(\lambda_+ - \lambda_-) - \tilde{h}_{N}(\lambda_+ +\lambda_-)\tilde{h}_{N-1}(\lambda_+ - \lambda_-) }{2\lambda_-}  \\
		&= \frac{e^{-\frac{N}{2}\lambda_-^2 +\frac{s^2}{2N} -\frac{N}{2} (\lambda_+ - \frac{s}{N})^2 } }{2\lambda_-}  \\
		&\qquad\qquad\qquad \cdot \left( \tilde{h}_{N-1}(\lambda_+  + \lambda_-) \tilde {h}_{N}(\lambda_+ -\lambda_-) - \tilde{h}_{N}(\lambda_+ + \lambda_-) \tilde{h}_{N-1}(\lambda_+ - \lambda_-)  \right)
		\end{split}\]
		Hence \[ \int_{-\infty}^{\infty} e^{s\lambda_+} K_N(\lambda_+ +\lambda_-,\lambda_+ -\lambda_-) \, \mathrm{d}\lambda_+  = \frac{e^{-\frac{N}{2}\lambda_-^2 +\frac{s^2}{2N} } }{2\lambda_-} \left(I^{N-1,N}_{\lambda_- +\frac{s}{N} , -\lambda_- +\frac{s}{N}} - I^{N-1,N}_{-\lambda_- +\frac{s}{N} , \lambda_- +\frac{s}{N}}\right) ~, \]
		where we defined \[ I^{k,l}_{a,b}= \int_{-\infty}^{\infty} e^{-\frac{Nx^2}{2}} \tilde{h}_k(x+a) \tilde{h}_l(x+b) \mathrm{d} x~.\] 
		Now we calculate $ I^{N-1,N}_{a,b} $  using the translation formula (see \thref{TransFor}) \[ h_n(x+a) = \sum_{k=0}^{n} \binom{n}{k} N^k a^k h_{n-k}(x) \]  for the Hermite polynomials:
		\begin{align*}
		I^{N-1,N}_{a,b} &= \int_{-\infty}^{\infty} e^{-\frac{Nx^2}{2}} \tilde{h}_{N-1}(x+a) \tilde{h}_N(x+b) \mathrm{d} x \\
		&= \frac{1}{ \sqrt{\frac{2\pi}{N}} (N-1)! N^N} \\
		&\qquad\quad\int_{-\infty}^{\infty} e^{-\frac{Nx^2}{2}} \left( \sum_{k=0}^{N-1} \frac{N^k a^k (N-1)!}{(N-1-k)! k!}h_{N-1-k}(x) \right) \left( \sum_{l=0}^N \frac{N^l b^l N!}{(N-l)! l!}h_{N-l}(x) \right) \mathrm{d}x \\
		&= \frac{1}{ \sqrt{\frac{2\pi }{N}} N^N} \sum_{k=0}^{N-1} \frac{N^k N^{k+1} a^k b^{k+1} N!} {\left( (N-k-1)! \right)^2 k! (k+1)!} (N-k-1)! N^{N-k-1} \sqrt{\frac{2\pi}{N}} \\
		&= \sum_{k=0}^{N-1} \frac{N^k a^k b^{k+1} N!}{(N-k-1)! k! (k+1)!}  ~.
		\end{align*}	
		Hence, \[I^{N-1,N}_{a,b}-I^{N-1,N}_{b,a} = \sum_{k=0}^{N-1} \frac{N^k a^k b^k (b-a) N!}{(N-k-1)! k! (k+1)!}\] and 
		\begin{align*}
		&\int_{-\infty}^{\infty} e^{s\lambda_+}K_N(\lambda_+ + \lambda_-,\lambda_+ - \lambda_-)\mathrm{d}\lambda_+ \\ &=
		\frac{e^{-\frac{N}{2}\lambda_-^2 +\frac{s^2}{2N} } }{2\lambda_-} \left(I_{\lambda_- +\frac{s}{N} , -\lambda_- +\frac{s}{N}} - I_{-\lambda_- +\frac{s}{N} , \lambda_- +\frac{s}{N}}\right) \\
		&= \frac{e^{-\frac{N}{2}\lambda_-^2 +\frac{s^2}{2N} } }{2\lambda_-} \sum_{k=0}^{N-1} \frac{ 2\lambda_- N^k \left(\frac{s^2}{N^2}-\lambda_-^2 \right)^k  N!}{(N-k-1)! k! (k+1)!} \\
		&= e^{-\frac{N}{2}\lambda_-^2 +\frac{s^2}{2N} } \sum_{k=0}^{N-1} \frac{ \left( \frac{s^2}{N}-N\lambda_-^2\right)^k N!}{(N-k-1)! k! (k+1)!}  ~.
		\end{align*}
		
		Recalling \eqref{eq:confl}, we identify the sum as $N\cdot {}_1F_1\left(1-N ; 2 | N\lambda_-^2 -\frac{s^2}{N} \right)$ and conclude that 
		\[
		\int_{-\infty}^{\infty} e^{s\lambda_+} K_N(\lambda_+ + \lambda_-,\lambda_+ -\lambda_-) \mathrm{d}\lambda_+ =  Ne^{-\frac{N}{2}\lambda_-^2 +\frac{s^2}{2N} } {}_1F_1\left(1-N ; 2 | N\lambda_-^2 -\frac{s^2}{N} \right) ~.
		\]
		
	\end{proof}

	\subsection{Differential equation for $p_N(\lambda)$ }
	
	As a corollary of \thref{FourierTrans} we have
	\begin{cor} [ {\cite[Lemma 2.1]{DiffEq}}]
		The mean eigenvalue density function of the $ \mathrm{GUE} $ satisfies the equation
		\begin{equation}\label{eq:diffeq}
		\frac{1}{N^2} p_N''' (\lambda) + (4-\lambda^2) p_N'(\lambda) +\lambda p_N(\lambda) =0 ~.
		\end{equation}
	\end{cor}
	
	\begin{proof}

	The confluent hypergeometric function $ {}_1F_1\left(a;c|x\right)$ satisfies the differential equation $ xf^{''} + (c-x) f^{'} -af =0 $ . Therefore $ g(x) = e^{\frac{x^2}{2}} {}_1F_1\left(a;c|-x^2\right) $ satisfies  
	\[xg''(x) + (2c-1) g' (x) -(x^3+2cx-4ax) g(x) =0 ~.\] We substitute $ a=1-N $ and $ c=2$ to conclude that the bilateral Laplace transform of $p_N(\lambda) $ , $ \hat{p}_N(x) = \int_{-\infty}^{\infty} e^{x\lambda} p_N (\lambda) \mathrm{d}\lambda = g(\frac{x}{\sqrt{N}}) $, satisfies 
	\begin{equation}
	 x\hat{p}_N''(x)  + 3 \hat{p}_N'(x) - (\frac{1}{N^2}x^3+4x) \hat{p}_N(x) = 0 ~.
	\end{equation}
Taking the inverse bilateral Laplace transform (defined since $p_N(t) $ decays as a polynomial times a gaussian as $t \to \pm \infty $) we conclude that
	\begin{equation}\label{DifEq}\begin{split} 
	0 &= -\lambda^2 p_N' (\lambda) -2\lambda p_N(\lambda) +3 \lambda p_N(\lambda) +\frac{1}{N^2} p_N'''(\lambda) + 4p'_N(\lambda)   \\
	&= \frac{1}{N^2} p_N''' (\lambda) + (4-\lambda^2) p_N'(\lambda) +\lambda p_N(\lambda) ~.
	\end{split}\end{equation}
	
	\end{proof}
	
	\begin{rmk}
	Using \thref{FourierTrans}, one can obtain a system of two partial differential equations for $ K_N (\lambda_1 ,  \lambda_2) $: one from the differential equation of the hypergeometric function, and one from the hyperbolic  symmetry between $\lambda_-$ and $s$.
	\end{rmk}
	
	\subsection{$\frac{1}{N}$ expansion of the bilateral Laplace transform} \label{Laplace Expansion}
	
	Let us prove \eqref{LapExpStatement}. 
	By \eqref{LaplaceExpression}, we have
	
	\begin{equation}\label{eq:ltr}
	\int e^{st} p_N(t) \mathrm{d}t =  e^{\left(\frac{s^2}{2N}\right)}\sum_{k=0}^\infty \frac{\left(1-\frac{1}{N}\right)\left(1-\frac{2}{N}\right) \dots \left(1-\frac{k}{N}\right) }{k!(k+1)!} s^{2k} 
	\end{equation}
	
	Recalling the definition of the Stirling numbers,
	\[ x(x-1)\cdots (x-n+1) = \sum_{k=0}^{n} (-1)^{n-k} \left[ n \atop k \right] x^k ~, \]
	we have
	\begin{align*}
	(\ref{eq:ltr})&= e^{\left(\frac{s^2}{2N}\right)}\sum_{k=0}^\infty \frac{ \sum\limits_{l=0}^{k+1} (-1)^{k+l+1} \left[k+1 \atop l \right] N^{l-k-1}}{k!(k+1)!} s^{2k} 
	\\
	&= e^{\left(\frac{s^2}{2N}\right)}\sum_{k=0}^\infty \frac{ \sum\limits_{l=0}^{k+1} (-1)^{l}\left[k+1 \atop k+1-l\right] N^{-l} }{k! (k+1)!} s^{2k} ~.
	\end{align*}
	
	This series converges absolutely uniformly since the Stirling number $ \left[k+1 \atop k+1-l\right]$ counts the number of permutations of $k+1$ elements with $ k+1-l$ disjoint cycles and thus satisfy $ \left[k+1 \atop k+1-l\right] \leq (k+1)! $ . Hence we may change the order of summation to get, 	
	\[
	\int e^{st} p_N(t) \mathrm{d}t = e^{\left(\frac{s^2}{2N}\right)} \sum_{l=0}^\infty (-1)^{l} \left(\sum_{k=0}^\infty  \frac{ \left[ k+1 \atop k+1-l \right] s^{2k}}{k!(k+1)!}  \right) N^{-l}~.
	\]	
	Both factors are analytic in $ \bar{\mathbb{C}} \backslash 0 $, hence the product is analytic at infinity and the expansion can be calculated by multiplying the expansions of the two factors. 
	\begin{rmk}
		It can be shown that $e^{ \frac{s^2}{2N} } {}_1F_1\left(1-N ; 2 | -\frac{s^2}{N} \right)$ is an even function of $N$ and hence the series contains only even powers of $ \frac{1}{N} $ .  
	\end{rmk}
	 
	\section{The operator $T$}
	\label{Operator}
	
	\subsection{Definition of the operator} \label{OneTPoly}
	
	Our goal is to construct an operator $T$ acting on the space of polynomials such that 
	\begin{equation} \label{OneTremark}
	\int_{- \infty}^{\infty} g(t)p_N(t) \mathrm{d} t  =  \frac{1}{2\pi } \int_{-2}^{2} g(s) \sqrt{4-s^2} \mathrm{d}s + \frac{1}{N^2} \int_{\mathbb{R}} [Tg](t) p_N (t) \mathrm{d}t ~.
	\end{equation}
	holds for all $g$ in a space of functions containing the polynomials. By the differential equation \ref{eq:diffeq} we have 
	\begin{equation*}
	0 = \int_{\mathbb{R}} f(t) \left[N^{-2}p_N'''(t) + (4-t^2) p_N'(t) +tp_N(t) \right]\, \mathrm{d}t 
	\end{equation*}
	and therefore, by integration by parts ,
	\begin{equation*}
	\int_{\mathbb{R}} \left[(t^2-4) f'(t) + 3t f(t) \right] p_N(t) \mathrm{d}t = \frac{1}{N^2} \int_{\mathbb{R}} f'''(t) p_N(t) \mathrm{d}t
	\end{equation*}
	Guided by this relation, we will investigate the action of the linear operators $ S,T $ on the space of polynomials, where $ Sg $ is the unique solution to the equation 
	
	\begin{equation} \label{S defenition}
	g(t)=\frac{1}{2\pi} \int_{-2}^{2}g(s)\sqrt{4-s^2}ds + (t^2-4) \left[Sg\right]'(t) + 3t \left[Sg\right](t) 
	\end{equation}
	and $ Tg = (Sg)''' $. Thus we obtain \ref{OneTremark}, (cf. \cite[Theorem 3.5]{HT} ).
	
	\subsection{The action on the space of polynomials} 
	
	Our goal in this part is to calculate the action of the operators $S$ and $T$ and to give bounds on their operator norms.\\	Since $S$ applied to a constant gives zero, $Sg$ is in fact dependent only on $g'$, so if we write $f=Sg$ we get
	\[
	g'\left(t\right) = \left(t^{2}-4\right) f''\left(t\right) +
	5tf'\left(t\right) + 3f\left(t\right) ~.
	\]
	
	We will look for formal eigenfunctions of this operator, i.e.\ polynomials for which $ Sg = \lambda g'$. Denoting $ g'=f$ the equation for $f$ is
	\begin{equation}
	\left(t^{2}-4\right)\partial_{tt} f+5t\partial_{t}f= \left(\frac{1}{\lambda} -3 \right) f
	\end{equation}
and its $C^{\infty }\left( \mathbb{R} \right) $ solutions are the Gegenbauer polynomials  $ f_{n} \left( t \right) = C_{n}^{\left( 2 \right)}(\frac{t}{2})  $ (see Appendix~\ref{a:G}), with $\lambda_n = (n+2)^2-1$ for $n=0,1,2,\dots $ . Those polynomials satisfy $ f_{n+1}' - f_{n-1}' = (n+2) f_{n} $, so the differentiation operator is represented in the basis $\{f_n\}_n$ by the matrix	
	\[
	D= 
	\begin{pmatrix}
		0      & 2 & 0 & 2 &  0 & 2 & \dots  \\
		0      & 0 & 3 & 0 & 3 &  0 & \dots  \\
		0      & 0 & 0 & 4 & 0 &  4  & \dots\\
		0       & 0 & 0 & 0 & 5 & 0 &\\ 
		\vdots &   &   &   &   & &\ddots
	\end{pmatrix}
	\]
	\begin{cor} In the basis $\{f_n\}_n$, $ S=H D$ and $T=D^3 HD$, 
	where $D$ is as above and $H_{mn}=\delta_{mn}\frac{1}{(n+2)^2-1} $ .\end{cor}
	
	\subsection{Bounds on the operator norm}
	Entire functions of order two and finite type can be written as convergent sums of Gegenbauer
	polynomials. This is the content of the next lemma
	
	\begin{Lemma}
		If $f$ is entire function of order two and finite type $\sigma_f$, then $ f = \sum_{n=0}^{\infty} a_n f_n $ , where the $ \mathrm{RHS} $ converges locally uniformly and 
		\begin{equation} \label{coefficients bound}  \sup \left\{ \abs{a_n} \left(\frac{n}{K} \right)^{\frac{n}{2}} \right\} < \infty ~, \quad K > 8 e \sigma_f~.
		\end{equation}
	\end{Lemma}
	
	\begin{proof}
		
		The function $f$ is continuous and hence it is in $\myLtwo $ and there, since $\{f_n\}_n$ forms an orthonormal basis, we can expand it in a series,
		\begin{equation}
		f=\sum_{n=0}^{\infty} a_n f_n ~,
		\end{equation}
		where by \eqref{normalization} 
		\[  2\pi(n+1)(n+3) a_n = \int_{-2}^{2} f(x) f_n(x) \left( 4-x^2 \right)^\frac{3}{2} \mathrm{d}x \] 
		Using that $f_n$ is orthogonal to all polynomials of degree less than $n$ and replacing $f$
		by its Taylor expansion around zero, we have
		\begin{align} \label{gegenbauer coefficient calculation} 
		2\pi (n+1)(n+3) \abs{a_n} &=  \abs{ \int_{-2}^{2} f(x) f_n(x) \left( 4-x^2 \right)^\frac{3}{2} \mathrm{d}x }  \nonumber \\
		&=  \abs{ \int_{-2}^{2} \sum_{k=0}^{\infty} \alpha_k x^k f_n(x) \left( 4-x^2 \right)^\frac{3}{2} \mathrm{d}x } \nonumber \\
		&=  \abs{ \sum_{k=n}^{\infty} \alpha_k \int_{-2}^{2} x^k f_n(x) (4-x^2)^\frac{3}{2} \mathrm{d}x } ~. \end{align}
		when $\alpha_n = f^{(n)}(0)/n!$ . Expanding it as a contour integral over a circle of radius $r$ and optimizing over $r$, we have (cf.\ \cite[Section I.2]{Lev}), for any $ \sigma > \sigma_f $,
		\begin{equation}\label{eq:cauchy} |\alpha_n| = \abs{\frac{f^{(n)}(0)}{n!}}\leq C_f \left(\frac{2e\sigma}{n}\right)^{\frac{n}{2}}~.\end{equation}
		Inserting this bound into \eqref{gegenbauer coefficient calculation} and using \eqref{normalization} we obtain:
		\begin{align*}
		2\pi (n+1)(n+3)|a_n| &\leq \sum_{k=n}^{\infty} C_f \left(\frac{2e\sigma}{k}\right)^{\frac{k}{2}}  \int_{-2}^{2} \abs{x^k} \abs{f_n(x)} (4-x^2)^\frac{3}{2} \mathrm{d}x \\
		& \leq \sum_{k=n}^{\infty}C_f \left(\frac{2e\sigma}{k}\right)^{\frac{k}{2}}  \cdot 2^k \cdot
		\left\{ 6\pi \int_{-2}^{2}  \abs{f_n(x)}^2 (4-x^2)^\frac{3}{2} \mathrm{d}x\right\}^{\frac12} \\
		&=\sum_{k=n}^{\infty}C_f \left(\frac{8e\sigma}{k}\right)^{\frac{k}{2}} \sqrt{6\pi\cdot 2\pi (n+1)(n+3)} ~,		\end{align*}
		hence \[ |a_n| \leq C_f' \left(\frac{8e\sigma}{n}\right)^{\frac{n}{2}}~.\]
		This implies  \eqref{coefficients bound} holds. Finally using Lemma~\ref{bounds fn}, we have for $\abs{z} \leq r$, $r > 3$,
		\[ \abs{a_n f_n\left(z\right)}  \leq 2C_f' \left(\frac{8e\sigma r^2}{n}\right)^{\frac{n}{2}} ~, \]
		so $\sum_{n=0}^{\infty} a_n f_n(z)  $  is analytic and hence $f(z) = \sum_{n=0}^{\infty} a_n f_n(z) $ and the RHS converges locally uniformly.
		
	\end{proof}
	
	Hence we can define norms of an entire functions of order two and finite type, $ f = \sum_n a_nf_n $, by 
	\begin{definition}
		\[ \norm{f}_{c,K,n} = \left( \frac{n}{K} \right)^{cn} \left|a_n \right|~, \quad \norm{f}_{c,K} = \sup\left\{ \left\| f \right\|_{c,K,n} \Big| \, n\in \mathbb{N} \right\} ~. \]
	\end{definition}  

	Now, we will use this norms to show bounds on the operator $T$.
	
	\begin{Lemma} \thlabel{T bound}
		For any $ c \geq \frac{1}{2} $ , $K>0$ the operator $T$ is bounded with respect to the norm $ \norm{\cdot}_{c,K} $.
	\end{Lemma}
	
	\begin{proof}
	
	We denote $ \norm{f}_{c,K,n}^+ = \underset{m \geq n}{\sup} \{ \norm{f}_{c,K,m} \} ~, $ and then
	\begin{align*}
	\norm{Df}_{c,K,n} & \leq \left( \frac{n}{K} \right)^{cn} \sum_{m>n} \left| a_m \right| (n+3)  \\
	& \leq \norm{f}_{c,K,n+1}^+ \left( \frac{n}{K} \right)^{cn} \sum_{m>n} \frac{4m}{m^{cm}} K^{cm} \\
	& \leq \norm{f}_{c,K,n+1}^+ C_{c,K} n^{1-c} ~,
	\end{align*}
	since the sum is dominated by the first term. Therefore \[ \norm{Tf}_{c,K,n} \leq \norm{HDf}_{c,K,n+3}^{+} \frac{C_{c,K}^3}{n^{3c-3}} \leq \norm{Df}_{c,K,n+3}^{+} \frac{C_{c,K}^3}{n^{3c-1}} \leq	\norm{f}_{c,K,n+4}^{+} \frac{C_{c,K}^4}{n^{4c-2}} ~, \] and in conclusion, for $ c \geq \frac{1}{2} $ ,
	\begin{equation}
	\norm{Tf}_{c,K} \leq C_{c,K}^4 \norm{f}_{c,K} ~.
	\end{equation}
	Hence the operator $T$ is bounded with respect to the norm $ \norm{\cdot}_{c,K} $ for each $c\geq \frac{1}{2} $.
	\end{proof}
	We conclude that the operator $T$ can be extended to the completion of the polynomials with respect to those norms (since $\norm{f}_{c_1,K_1} \leq \norm{f}_{c_2,K_2} $  whenever $ c_1 \leq c_2  $ and $ K_1 \geq K_2 $ the extensions with respect to all these norms are consistent).	
	
	\section{Proof of \thref{mainThm}} \label{proof}

%
	
		
	
	To prove Theorem 1 we will need the following preliminary result
	
	\begin{Lemma}  \thlabel{OneTlem} For any $K>0$ there exist $N_0(K), C_K $ such that for $N> N_0 $ 
	\[ \left| \int f(t) p_N(t) \, \mathrm dt \right| \leq C_K \norm{f}_{\frac12,K} \quad  \text{if} \quad \norm{f}_{\frac12,K} < \infty~. \]
			\end{Lemma}	
		
	\begin{rmk}
		The bound on the integral (i.e. $ C_K $) is dependent in $ N$. The proof of \thref{mainThm} does not require it to be bounded.
	\end{rmk}
	
	\begin{proof}
		Let $ f = \sum_{n=0}^{\infty} a_n f_n $ with  $ \abs{a_n} \leq \left( \frac{n}{K} \right)^{-\frac{n}{2}} \norm{f}_{\frac12,K} $. Then
		
		\begin{equation} \label{Trace dif}
		\abs{\int f(t) p_N(t) \mathrm{d}t} \leq \sum_{n=0}^{\infty} \left( \frac{n}{K}\right)^{-\frac{n}{2}} \norm{f}_{\frac12,K}  \int_{\mathbb{R}} \abs{f_n(t)}p_N(t) \mathrm{d} t ~.
		\end{equation}
		Now, to complete the proof it is enough to show that the sum is convergent. First, using integration by parts,
		\begin{align*}
		\int_{\mathbb{R}} \abs{f_n(t)}p_N(t) \mathrm{d}t &\leq 12 \cdot 6^n + 2 \int_{3}^{\infty} \abs{f_n(t)} p_N(t) \mathrm{d}t \\
		&\leq 12\cdot 6^n + 4 \cdot 3^n + 2\int_{3}^{\infty} \abs{f_n'(t)}Ne^{-\frac{N(t-2)^2}{2}} \mathrm{d}t ~,
		\end{align*}
		where we used  the fact that $ \mathbb{P}(\lambda_N > 2+r ) \leq N e^{-\nicefrac{Nr^2}{2}} $ (the bound is known to be true even without the prefactor $N$, however, we  use the version with the prefactor which was derived in \cite{large eig} from the bilateral Laplace transform computation (\ref{LaplaceExpression})). By \thref{bounds fn} we have:
		\begin{align*}
		\int_{3}^{\infty} \abs{f_n'(t)}e^{-\frac{N(t-2)^2}{2}} \mathrm{d}t 
		&\leq 3\int_{3}^{\infty} t^n e^{-\frac{N (t-2)^2}{2}} \mathrm{d}t \\ 
		&\leq 3\int_{0}^{\infty} (t+2)^n e^{-\frac{N t^2}{2}} \mathrm{d}t \\
		& \leq 3 \cdot 2^{n-1} \int_{0}^{\infty} (t^n+2^n) e^{ - N\frac{t^2}{2} } \mathrm{d}t  \\
		& \leq 3 \cdot 2^{n-1} \left( \frac{ \Gamma\left( \frac{n+1}{2} \right) }{N^{\frac{n+1}{2}}} + 2^n \sqrt{\frac{2\pi}{N}} \right) ~,
		\end{align*}
		and therefore the sum (\ref{Trace dif}) is convergent for $N > N_0(K) = \operatorname{const} \cdot K$.
		
	\end{proof}
	
	Now we can start the proof of \thref{mainThm}. \\
	Let $f$ be an entire function of order two and finite type. Using  \thref{OneTlem} we can approximate it by polynomials
	with respect to $\norm{\cdot}_{\frac12,K}$ for sufficiently large $K$ 
	and derive from \ref{OneTremark} that
		\begin{equation} 
		\int_{\mathbb{R}} f(t) p_N (t) \mathrm{d}t  =  \frac{1}{2\pi } \int_{-2}^{2} f(s) \sqrt{4-s^2} \mathrm{d}s + \frac{1}{N^2} \int_{\mathbb{R}} [Tf](t) p_N (t) \mathrm{d}t 
		\end{equation}
Iterating, we have:
	\begin{equation}\begin{split}
	&\int_{\mathbb{R}} f(t) p_N (t) \mathrm{d}t  = \frac{1}{2\pi}  \sum\limits_{k=0}^{m} \frac{1}{N^{2k}} \int_{-2}^{2} \left[T^kf\right] (t) \sqrt{4-t^2} \mathrm{d}t \\
	&\qquad\qquad\qquad+ \frac{1}{N^{2m+2}} \int_{\mathbb{R}} \left[ T^{m+1}f \right] (t) p_N (t) \mathrm{d}t~.
	\end{split}\end{equation}
	Finally, 
	\begin{align*}
	 \abs{\frac{1}{N^{2m}} \int_{\mathbb{R}} \left[ T^{m}f \right] (t) p_N (t) \mathrm{d}t }
	 &\leq  \frac{1}{N^{2m}}\norm{T^m f}_{\frac12, K} C_K   \\
	 &\leq \frac{1}{N^{2m}}\norm{T}_{\frac12, K}^{m} C_K \norm{f}_{\frac12, K} \to 0
	\end{align*}
	as $m \to \infty$, for sufficiently large $N$, where the first inequality follows from \thref{OneTlem} and the second from \thref{T bound}. This completes the proof of \thref{mainThm}.\qed

\begin{appendices}	

\section{Gegenbauer polynomials } \label{a:G}

	The Gegenbauer polynomials for $ \lambda > -\frac{1}{2} $ are defined by
	\[ C_n ^{( \lambda )} (x) = \binom{n+2\lambda -1}{n} {}_2F_1\left(-n , n+2\lambda ; \lambda + \frac{1}{2}\, \mid \,\frac{1}{2} (1-x)  \right) \] where for $ c \neq 0 , -1, -2 , \dots $
	\[ {}_2F_1(a , b ; c | x) = \sum_{k=0}^{\infty} \frac{ a(a+1) \cdots (a+k-1) b (b+1) \cdots (b+k-1) }{c(c+1) \cdots (c+k-1) k!} x^k  \]
	is the Gauss hypergeometric function. They are the orthogonal polynomials (see \cite[Chapter IV]{GegPol}) with respect to $(1-x^2)^{\lambda -\frac12}$ , i.e.\  
	\[
	\int_{-1}^{1}\left(1-x^{2}\right)^{\lambda-\frac{1}{2}}C_{n}^{\left(\lambda\right)}\left(x\right)C_{m}^{\left(\lambda\right)}\left(x\right)\mathrm{d}x=2^{1-2\lambda}\pi\frac{\Gamma\left(n+2\lambda\right)}{\left(n+\lambda\right)\Gamma^{2}\left(\lambda\right)\Gamma\left(n+1\right)}\delta_{n,m} ~,
	\]
Therefore for $ f_n(x) = C^{(2)}_n \left(\frac{x}{2}\right)  $ we have 
	\begin{equation}
	\int_{-2}^{2} f_n(x) f_m(x) (4-x^2)^\frac{3}{2} \mathrm{d}x 
	= 2\pi (n+3)(n+1) \delta_{n,m} ~. \label{normalization} 
	\end{equation}

	\begin{Lemma} \thlabel{bounds fn}
		For $\abs{z}=r>3$ it holds that $\abs{f_n(z)} \leq 2r^n $ , $\abs{f_n'(z)} \leq 3 r^n $, while for $\abs{z}\leq 3 $ it holds that $\abs{f_n(z)} \leq 2\cdot 3^n$ , $ \abs{f_n'(z)} \leq 3 \cdot 3^n $.
	\end{Lemma}
\begin{proof}
	
	Recall that the Chebyshev polynomials of the first kind are defined by $ T_n (x) ={}_2F_1\left(-n,n;\frac{1}{2} | \frac{1}{2}(1-x)
	\right) $ and satisfy \begin{equation}\label{eq:cheb}T_n(\cos\theta) = \cos(n\theta) \, , \, T_n(\cfrac{w+\frac{1}{w}}{2}) = \cfrac{w^n+\frac{1}{w^n}}{2} ~. \end{equation}
	Let us show that \[ \frac{\partial^2}{\partial x^2}T_{n+2}(\frac{x}{2}) = \frac{n+2}{2} f_n(x) ~, \] where $ f_n(x) = C^{(2)}_n \left(\frac{x}{2}\right)  $ as above. First, 
	\begin{align*}
		&\frac{\partial^{2}}{\partial x^{2}} {}_{2}F_{1}\left(-m,m;\frac{1}{2}\Big|\frac{1-x}{2}\right) \\
		&\qquad= \frac{1}{4}\frac{-m^{2}} {\nicefrac{1}{2}} \frac{\left(-m+1\right)\left(m+1\right)} {\nicefrac{3}{2}} \, _{2}F_{1} \left(-m+2,m+2;\frac{5}{2}\Big|\frac{1-x}{2}\right) \\
		 &\qquad= \frac{m^{2}\left(m^{2}-1\right)}{3} \,_{2}F_{1} \left(-m+2,m+2;\frac{5}{2}\Big|\frac{1-x}{2}\right) ~.
	\end{align*}
Therefore,
\begin{align*}
	\frac{\partial^{2}}{\partial x^{2}} T_{n+2}\left( \frac{x}{2} \right) &= \frac{(n+2)^2(n^2+4n+3)}{12} {}_2F_1 \left(-n,n+4;\frac{5}{2}\Big|\frac{2-x}{4}\right) \\
	&= \frac{(n+2)^2(n+3)(n+1)}{12} {}_2F_1 \left(-n,n+4;\frac{5}{2}\Big|\frac{2-x}{4}\right) \\
	&= \frac{(n+2)^2(n+3)(n+1)}{12} \binom{n+3}{n}^{-1} f_n(x) \\
	&= \frac{n+2}{2} f_n(x)	~,
	\end{align*}
as claimed. 
	
	Let $|\zeta| = \rho>1$. Choose $\xi = \zeta \pm \sqrt{\zeta^2-1}$ so that $|\xi| \geq \rho$; then by (\ref{eq:cheb})
	\[ \abs{T_n(\zeta)}=\abs{T_n\left(\frac{\xi+\xi^{-1}}{2}\right)} =\abs{\frac{\xi^n + \xi^{-n}}{2}} \leq \frac{\rho^n+\rho^{-n}}{2} ~.\]
	By the Cauchy formula we get that for $\abs{z}=r > 3 $ : 
	\begin{equation*} 
		\abs{f_n(z)} = \frac{2}{n+2} \abs{\frac{\partial^2}{\partial x^2 } T_{n+2} \left(\frac{x}{2}\right)} \leq \frac{4}{n+2} \left(\frac{r}{2}+1\right)^n \leq 2r^n
	\end{equation*}
	 and for $r \leq 3 $ , by the maximum principle $\abs{f_n(z)} \leq 2 \cdot 3^n  $ . Similarly, for $r \geq 3 $, $\abs{f_n'(z)} \leq 3r^n $ and for $r \leq 3 $, $ \abs{f_n'(z)} \leq 3\cdot 3^n $.
		
\end{proof}

\section{Hermite Polynomials} \label{HerProp}

	We use the definition of the Hermite polynomials in \cite{Intro} ,
	\begin{equation} \label{Hermite Def}
	h_k(x)=(-1)^k e^{N\frac{x^2}{2}} \frac{\mathrm{d}^k}{\mathrm{d}x^k} e^{-N\frac{x^2}{2}},
	\end{equation}
	and  the normalized Hermite polynomials $ \tilde{h}_k(x)= \frac{1}{\left[k! N^k \sqrt{\frac{2\pi}{N}}\right]^\frac{1}{2}} h_k(x) $. They are orthogonal polynomials with respect to $ e^{-\frac{Nx^2}{2}} $, i.e.\  \[ \int_{-\infty}^{\infty} e^{-\frac{Nx^2}{2}}  \tilde{h}_k(x)  \tilde{h}_l(x) \mathrm{d}x = \delta_{k,l}\]
	
	\begin{rmk}
		The standard definition of the Hermite polynomials is  
		\[ H_k(x)=(-1)^k e^{\frac{x^2}{2}} \frac{\mathrm{d}^k}{\mathrm{d}x^k} e^{-\frac{x^2}{2}} \]
		\textup{(see \cite[Chapter V]{GegPol})} and they are orthogonal with respect to  $e^{-\frac{x^2}{2}} $. 
	\end{rmk}
	
	The main fact about the Hermite polynomials we will use is the translation formula
	
	\begin{Lemma}[Translation Formula] \thlabel{TransFor}
		$h_n(x+a) = \sum_{k=0}^{n} \binom{n}{k} N^k a^k h_{n-k}(x)$ .
	\end{Lemma}
	
	\begin{proof}
		The Hermite polynomials are orthogonal, therefore for $k \neq n-1$,
		
		\begin{align*}
		\int_{-\infty}^{\infty} e^{-N\frac{x^2}{2}} h_n'(x) h_k(x) \mathrm{d}x &= \int_{-\infty}^{\infty} (-1)^k \frac{\mathrm{d}^k}{\mathrm{d}x^k}  \left[e^{-N\frac{x^2}{2}} \right] h_n'(x)\mathrm{d}x \\
		&= (-1)^{k+1} \int_{-\infty}^{\infty} \frac{\mathrm{d}^{k+1}}{\mathrm{d}x^{k+1}}  \left[e^{-N\frac{x^2}{2}} \right] h_n(x)\mathrm{d}x  \\
		&=0 ~.
		\end{align*}
		
		So, $h_n' = ch_{n-1} $ for some constant $c$ , and by comparing the leading order coefficient we conclude $ h_n'(x) = Nn h_{n-1}(x) $. \\
		Now, by induction, $ h_n^{(k)} = N^k \frac{n!}{(n-k)!} h_{n-k}(x) $. Substituting this into the Taylor expansion, $h_n(x+a) = \sum_{k=0}^{n} \frac{1}{k!} a^k h_n^{(k)}(x) $, we get that \[h_n(x+a)= \sum_{k=0}^{n} \binom{n}{k} N^k a^k h_{n-k}(x) ~. \]
	\end{proof}
	
\end{appendices}

\end{document}